\documentclass[12pt]{article}
\usepackage[titletoc,title]{appendix}
\usepackage{amsmath,amsfonts,amssymb,bm,amsthm,amscd}
\usepackage[colorlinks,citecolor=blue,linkcolor=blue]{hyperref}
\numberwithin{equation}{section}
\newtheorem{theorem}{Theorem}[section]

\newtheorem{proposition}[theorem]{Proposition}

\theoremstyle{definition}

\newtheorem{assumption}[theorem]{Assumption}
\newtheorem{remark}[theorem]{Remark}

\newtheorem{question}{Question}

\numberwithin{equation}{section}

\newcommand{\CP}{\mathbb{C}\mathrm{P}}
\newcommand{\RP}{\mathbb{R}\mathrm{P}}
\newcommand{\Gr}{\operatorname{Gr}}
\newcommand{\V}{\operatorname{V}}

\usepackage[T1]{fontenc}
\usepackage[top=1in, bottom=1in, left=1in, right=1in]{geometry}
\usepackage{fancyhdr}
\usepackage{enumerate}
\usepackage{xcolor,cancel}
\usepackage{relsize}

\usepackage{tikz}
\usetikzlibrary{arrows.meta}

\allowdisplaybreaks

\begin{document}

\title{Topological obstructions
\\ to the diagonalisation of pseudodifferential systems}
\author{\normalsize Matteo Capoferri\thanks{MC:
School of Mathematics,
Cardiff University,
Senghennydd Rd,
Cardiff CF24~4AG,
UK. \emph{Current address}:
Department of Mathematics, Heriot-Watt University, EH14 4AS Edinburgh, UK;
m.capoferri@hw.ac.uk,
\url{http://mcapoferri.com}.
}
\and
\normalsize Grigori Rozenblum\thanks{GR:
Department of Mathematical Sciences,
Chalmers University of Technology, 
Sweden \emph{and} The Euler International Mathematical Institute, Saint Petersburg \emph{and} Sirius University, Sochi, Russia;
grigori@chalmers.se.
}
\and
\normalsize Nikolai Saveliev\thanks{NS:
Department of Mathematics, 
University of Miami, 
PO Box 249085 Coral Gables, 
FL 33124, USA;
saveliev@math.miami.edu.
}
\and
\normalsize Dmitri Vassiliev\thanks{DV:
Department of Mathematics,
University College London,
Gower Street,
London WC1E~6BT,
UK;
D.Vassiliev@ucl.ac.uk,
\url{http://www.ucl.ac.uk/\~ucahdva/}.
}}

\renewcommand\footnotemark{}


\maketitle
\begin{abstract}
Given a matrix pseudodifferential operator on a smooth manifold, one may be interested in diagonalising it by choosing eigenvectors of its principal symbol in a smooth manner. We show that diagonalisation is not always possible,  on the whole cotangent bundle or even in a single fibre.  We identify global and local topological obstructions to diagonalisation and examine physically meaningful examples demonstrating that all possible scenarios can occur.

\


\

{\bf 2020 MSC classes: }
primary 
58J40; 
secondary 
35G35, 	
35J46, 	
35J47,  	
35J48. 	

\end{abstract}


\allowdisplaybreaks

\section{Statement of the problem}
\label{Statement of the problem}

Diagonalisation is often a useful approach to \textcolor{black}{recasting} matrix operators appearing in analysis \textcolor{black}{and} mathematical physics in a form that can be more easily analysed.  Its effectiveness can be already appreciated at the level of operators on finite-dimensional vector spaces, where it manifests itself in many guises,  not least in the various formulations of the Spectral Theorem.

For partial differential or, more generally,  pseudodifferential matrix operators on manifolds, the problem of diagonalisation can effectively be reduced to the diagonalisation of the principal symbol of the operator at hand --- a smooth matrix-function on the cotangent bundle ---, and more precisely to the existence of globally defined eigenvectors thereof.  Indeed,  as soon as one can \emph{globally} diagonalise the principal symbol in a smooth manner,  \textcolor{black}{many} approaches to \textcolor{black}{achieving} (block) diagonalisation in various settings are available in the literature \cite{taylor_diag, cordes3,cordes,helffer_sj,BR99,PST03, cordes2, cuenin, diagonalization}.

However, there are, in general, obstructions of topological nature that prevent one from choosing smooth global eigenvectors of the principal symbol. Remarkably, such obstructions may be present even (i) for operators acting on \emph{trivial} vector bundles and (ii) in the cotangent fibre \emph{at a single point} of the base manifold. The goal of the current paper is to examine the issue of topological obstructions and provide necessary and sufficient conditions for the diagonalisation of pseudodifferential \textcolor{black}{matrix operators} on manifolds, in a way that is self-contained and accessible to a wide readership, including researchers with a background in the analysis of PDEs.

The issue of diagonalisation of matrix-functions over topological spaces and its relation with the topology of the underlying space has, of course, been studied before. For example, in 1984 Kadison \cite{kadison} provided an explicit $2\times2$ normal continuous matrix-function on $\mathbb{S}^4$ that is not globally diagonalisable and asked the question of what topological \textcolor{black}{properties} of the underlying space guarantee diagonalisability of a $2\times 2$ normal continuous matrix-function.  The same year,  Grove and Pedersen \cite{grove_pedersen} exhibited a rather exotic class of compact Hausdorff spaces on which the diagonalisability is guaranteed for \emph{all} normal matrix-valued functions. In the same paper, they also showed that all normal matrix-valued functions with simple eigenvalues on a 2-connected compact CW-complex are diagonalisable \cite[Theorem~1.4]{grove_pedersen}.  More recently,  Friedman and Park \cite{friedman_park} took Grove and Pedersen's analysis further,  investigating unitary equivalence classes of normal matrix-functions under the assumption of simple eigenvalues.  

The novelties of our work are as follows: (i) motivated by \textcolor{black}{the} pseudodifferential theory and applications to partial differential operators from mathematical physics \textcolor{black}{and geometry},  we examine the special case of smooth matrix-functions on the cotangent bundle of a manifold $M$; (ii) we formulate the problem in an operator-theoretic framework and in the language of mathematical analysis, thus making the paper accessible to a readership with little or no topological background; (iii) we discuss obstructions both to the global diagonalisation and to the diagonalisation in the cotangent fibre at a single point; (iv) we discuss explicitly numerous physically meaningful examples,  detailing,  for each of them, existence or absence of local and global topological obstructions.

\

Let $M$ be a connected closed \textcolor{black}{oriented} \textcolor{black}{smooth} manifold of dimension $d\ge 2$. Local coordinates on $M$ will be denoted by $x^\alpha\,$, $\alpha=1,\dots,d$, and coordinates in the cotangent fibre $T_x^*M$ by $\xi_\alpha\,$, $\alpha=1,\dots,d$. \textcolor{black}{Throughout the paper we adopt Einstein's summation convention over repeated indices. }

Let $E$ be a trivial $\mathbb{C}^m$--bundle over $M$ with $m\ge 2$. Let $A$ be a 
pseudodifferential operator of order $s\in \mathbb{R}$ acting on \textcolor{black}{the sections of} $E$ and let $A_\mathrm{prin}(x,\xi)$ be its principal symbol. This principal symbol is an invariantly defined $m\times m$ \textcolor{black}{smooth} matrix-function on $T^*M\setminus\{0\}$, positively homogeneous of degree $s$ in $\xi$.

We make the following assumptions.
\begin{assumption}
\label{assumption principal symbol}
The principal symbol $A_\mathrm{prin}(x,\xi)$ of $A$ is Hermitian.
\end{assumption}
\begin{assumption}
\label{assumption simple eigenvalues}
The eigenvalues of $A_\mathrm{prin}(x,\xi)$ are simple for all $(x,\xi)\in T^*M\setminus\{0\}$.
\end{assumption}

\begin{remark}
Let us emphasise that ellipticity, or indeed self-adjointness, are not needed to formulate the main results of this paper, as the examples from Sections~\ref{Local obstructions}--\ref{Neither local nor global obstructions} will demonstrate.  Note that Assumption~\ref{assumption principal symbol} is weaker \textcolor{black}{than} formal self-adjointness \textcolor{black}{of $A$ with respect to the inner product on sections of $E$ defined by
\begin{equation}
\label{inner product}
\langle u,v\rangle
:=
\int_M u^*(x)\,v(x)\,\rho(x)\,dx\,,
\end{equation}
where $\rho: M\to\mathbb{R}$ is some positive smooth density,  the star $*$ indicates Hermitian conjugation and $dx=dx^1\dots dx^d\,$.  Indeed, formal self-adjointness requires additional conditions on the lower order terms of the symbol}.  However,
in applications, including applications to spectral theory, one often assumes that $A$ is symmetric \textcolor{black}{(with respect to \eqref{inner product})} and elliptic,  namely,
\begin{equation*}
\label{ellipticity}
\det A_\mathrm{prin}(x,\xi)\ne0,
\qquad
\forall (x,\xi)\in T^*M\setminus\{0\}.
\end{equation*}
\textcolor{black}{In this case, } $A$ is \textcolor{black}{automatically} self-adjoint in the full operator theoretic sense \textcolor{black}{as} \textcolor{black}{an operator acting from the Sobolev space $H^s(M;E)$ to $L^2(M;E)$ with respect to the inner product \eqref{inner product}.}
This is a special case to which our results apply, subject to the validity of Assumption~\ref{assumption simple eigenvalues}.
\end{remark}

\color{black}
\begin{remark}
\label{remark smoothness eigenvalues}
Assumptions~\ref{assumption principal symbol} and~\ref{assumption simple eigenvalues} are enough to guarantee that the eigenvalues of $A_\mathrm{prin}$ are \emph{smooth} functions on $T^*M\setminus\{0\}$,  see, e.g., \cite[Section~7]{AKLM} and \cite{KMR} --- this is not the case if Assumption~\ref{assumption simple eigenvalues} on the simplicity of eigenvalues is dropped.
\end{remark}
\color{black}

We denote by $h^{(j)}(x,\xi)$ the eigenvalues of $A_\mathrm{prin}(x,\xi)$ and by 
\begin{equation*}
J\subset \mathbb{Z}, \qquad \# J=m,
\end{equation*}
the \textcolor{black}{index set for} $j$\footnote{\textcolor{black}{For the purposes of the current paper, the way in which $J$ is chosen is unimportant. There are, however,  circumstances --- for example when studying the spectrum of elliptic systems --- where it is convenient to choose the set $J$ in a particular way, see, e.g., \cite[Sec.~1]{part1}\cite[Sec.~1]{part2}. }}.
Let us denote by $P^{(j)}(x,\xi)$ the eigenprojection of $A_\mathrm{prin}(x,\xi)$ associated with the eigenvalue $h^{(j)}(x,\xi)$. It is easy to see that, for each $j\in J$, $P^{(j)}$ is a uniquely defined rank 1 (in view of Assumption~\ref{assumption simple eigenvalues}) smooth matrix-function on $T^*M\setminus \{0\}$. 
\textcolor{black}{Note that the eigenvalues $h^{(j)}(x,\xi)$ and the eigenprojections $P^{(j)}(x,\xi)$ are positively homogeneous in momentum $\xi$ of degree $s$ and zero,  respectively.}

It is natural to ask the following questions for each individual $j\in J$.

\begin{question}
\label{Q1}
For a given $x\in M$, can one choose an eigenvector $v^{(j)}(x,\xi)$ of $A_\mathrm{prin}\textcolor{black}{(x,\xi)}$ corresponding to the eigenvalue $h^{(j)}(x,\xi)$ smoothly for all $\xi\in T^*_xM\setminus \{0\}$?
\end{question}

\begin{question}
\label{Q2} 
Suppose Question~\ref{Q1} has an affirmative answer. Can one choose an eigenvector $v^{(j)}(x,\xi)$ of $A_\mathrm{prin}\textcolor{black}{(x,\xi)}$ corresponding to the eigenvalue $h^{(j)}(x,\xi)$ smoothly for all $(x,\xi)\in T^*M\setminus \{0\}$?
\end{question}

\color{black}
The goal of our paper is to answer Questions~\ref{Q1} and~\ref{Q2}. This will be done in full generality in Section~\ref{Main results}.

In Sections~\ref{Local obstructions}--\ref{Neither local nor global obstructions}, we will convert the abstract results of Section~\ref{Main results} into concrete calculations. We shall provide several explicit physically meaningful examples which demonstrate that, when it comes to topological obstructions, all possible scenarios can occur:
\begin{enumerate}[(i)]

\itemsep0em

\item Local obstructions --- massless Dirac operator (subsections~\ref{Massless Dirac operator in 3D} and~\ref{Top proof}) and the operator curl (subsection~\ref{The operator curl}) on a closed oriented 3-manifold;

\item Global obstructions but no local obstructions --- restriction of the massless Dirac operator to a 2-sphere (subsection~\ref{Restriction of the massless Dirac operator to the 2-sphere}) and an artificial example (subsection~\ref{An artificial example});

\item Neither local nor global obstructions ---  linear elasticity operator on an oriented Riemannian 2-manifold (subsection~\ref{Linear elasticity}) and the Neumann--Poincar\'e operator for 3D linear elasticity (subsection~\ref{NP example}). 

\end{enumerate}

Finally, in Section~\ref{Pseudodifferential operators with multiplicities} we will comment on possible generalisations.

\color{black}


\section{Main results}\label{Main results}

In this section we state and prove our main results. We are using \textcolor{black} {the notation} from the previous section.

\color{black}
Let us begin by observing that, even though the bundle $E$ is trivial, the range of the eigenprojection $P^{(j)}$ may define a nontrivial line bundle. Clearly, when there are no topological obstructions to the existence of a global eigenvector $v^{(j)}$, we have
\begin{equation*}
P^{(j)}(x,\xi)=\frac{v^{(j)}(x,\xi)\,[v^{(j)}(x,\xi)]^*}{[v^{(j)}(x,\xi)]^*\,v^{(j)}(x,\xi)}\,.
\end{equation*}
Note that $v^{(j)}$ is only defined up to a local gauge transformation $v^{(j)}\mapsto z^{(j)} v^{(j)}$, where $z^{(j)}:T^*M\setminus\{0\}\to \mathbb{C}^*$ is an arbitrary smooth function\footnote{In agreement with standard terminology in theoretical physics, here by `local' we mean that the value of $z^{(j)}$ depends on $(x,\xi)\in T^*M\setminus\{0\}$. The function $z^{(j)}$ itself is defined globally.}. Irrespective of whether the eigenvector $v^{(j)}$ is defined globally, we have a well-defined smooth map 
\begin{equation}\label{E:cp}
f^{(j)}: T^*M\setminus\{0\}\longrightarrow \CP^{m-1}
\end{equation}
to the complex projective space, sending $(x,\xi)$ to the complex line through the origin in $\mathbb C^m$ spanned by the vector $v^{(j)}(x,\xi)$. This map is  positively homogeneous of degree zero in momentum $\xi$. Choosing a smooth eigenvector $v^{(j)}(x,\xi)$ then amounts to finding a smooth lift of this map with respect to the canonical projection 
\begin{equation*}\label{E:bundle1}
p: \;\mathbb{C}^m\setminus\{0\} \longrightarrow \CP^{m-1}
\end{equation*}
or, after normalising our eigenvectors to have length one, with respect to the projection of the canonical circle bundle
\begin{equation}\label{E:bundle}
p:\;{\mathbb S}^{2m-1} \longrightarrow \CP^{m-1}
\end{equation}
sending a point on the unit sphere ${\mathbb S}^{2m-1} \subset \mathbb C^{m}\setminus\{0\}$ to the complex line through that point. The lift in question is the dotted arrow that makes the following diagram commute:

\begin{equation*}
\begin{tikzpicture}
\draw (4,1) node (a) {$T^* M \setminus \{0\}$};
\draw (8,3.2) node (b) {\quad $\mathbb{S}^{\,2m-1}$};
\draw (8,1) node (c) {$\CP^{m-1}$};
\draw[-{Stealth[length=2mm, width=2mm]},dashed](a)--(b) node {}; 
\draw[-{Stealth[length=2mm, width=2mm]}](a)--(c) node [midway,above](TextNode){$f^{(j)}$};
\draw[-{Stealth[length=2mm, width=2mm]}](b)--(c) node [midway,right](TextNode){$p$};
\end{tikzpicture}
\end{equation*}
\color{black}

\begin{theorem}
\label{main theorem 1}
One can choose an eigenvector $v^{(j)}(x,\xi)$ smoothly for all $(x,\xi)\in T^*M\setminus\{0\}$ if and only if either one of the following two equivalent conditions holds:
\begin{enumerate}
\item[(1)] the map \eqref{E:cp} induces a zero map $H^2 (\CP^{m-1}) \longrightarrow H^2 (T^*M \setminus \{0\})$ in cohomology, or 
\item[(2)] the Euler class $e \in H^2 (\CP^{m-1})$ of the circle bundle \eqref{E:bundle} pulls back to zero via \eqref{E:cp}.
\end{enumerate}
\end{theorem}

\begin{proof}
Finding a lift of the map \eqref{E:cp} is equivalent to finding a section of the principal circle bundle \textcolor{black}{$L$} over  $T^*M \setminus \{0\}$ obtained by pulling back the principal circle bundle \eqref{E:bundle} via the map $f^{(j)}$. As a principal bundle, \textcolor{black}{$L$} admits a section if and only if it is trivial, see \cite[Chapter 4, Corollary 8.3]{Husemoller}. The triviality of \textcolor{black}{$L$} is in turn equivalent to the vanishing of its Euler class,  see \cite[Chapter 20, Remark 6.2]{Husemoller}. The latter class equals the pull-back of the Euler class $e \in H^2 (\CP^{m-1})$ of the canonical bundle \eqref{E:bundle} via $f^{(j)}$.  Since $e$ generates $H^2 (\CP^{m-1})$, the proof \textcolor{black}{of the equivalence is complete.  It remains only to show that the lift is smooth (the topological arguments only guarantee a \emph{continuous} lift); but smoothness follows from the triviality of $E$ and \cite[Section~7]{AKLM},  \cite{KMR}}.
 \end{proof}

\begin{theorem}
\label{main theorem 2}
For a fixed point $x\in M$, one can choose an eigenvector $v^{(j)}(x,\xi)$ smoothly for all $\xi \in T^*_x M\setminus\{0\}$ if and only if 
\begin{enumerate}
\item[(1)] either $d \neq 3$, or
\item[(2)] $d = 3$ and the map $T^*_x M\setminus\{0\} \to \CP^{m-1}$ obtained by restricting the map \eqref{E:cp} to the fiber at $x \in M$ is homotopic to a constant map.
\end{enumerate}
\end{theorem}

\begin{proof}
The same argument as in the proof of Theorem \textcolor{black}{\ref{main theorem 1}} shows that, for a given $x \in M$, an eigenvector $v^{(j)}(x,\xi)$ can be chosen smoothly for all $\xi \in T^*_x M\setminus\{0\}$ if and only if the induced map 
\[
H^2(\CP^{m-1}) \longrightarrow H^2(T^*_x M\setminus\{0\})
\]
is zero \textcolor{black}{(smoothness is obtained \emph{a posteriori} by arguing as in the proof of Theorem~\ref{main theorem 1})}. Since $T^*_x M\setminus\{0\}$ has \textcolor{black}{a sphere} $\mathbb{S}^{\,d-1}$ as its deformation retract, the group $H^2 (T^*_x M \setminus\{0\})$ vanishes unless $d = 3$. In the case of $d = 3$, one only needs to show that any map $f: {\mathbb S}^2 \longrightarrow \CP^{m-1}$ inducing a zero homomorphism $H^2 (\CP^{m-1}) \to H^2 (\mathbb S^2)$ is homotopic to a constant map. By the cellular approximation theorem \textcolor{black}{\cite[Theorem~4.8]{hatcher}}, $f$ is homotopic to a map $h: {\mathbb S}^2 \to \CP^1$ into the 2-skeleton of $\CP^{m-1}$ with its standard cellular structure. In particular, $0 = h^*: H^2 (\CP^1) \to H^2 (\mathbb S^2)$. Since $\CP^1 = {\mathbb S^2}$, the result follows from the Hopf theorem classifying the homotopy classes of maps from a sphere to itself by their degree; \textcolor{black}{see \cite[Corollary 4.25]{hatcher}.}
\end{proof}

It is worth emphasising that Theorem~\ref{main theorem 2} singles out dimension $d=3$ as special.  This is very relevant in applications, as dimension three is the natural setting of a large number of physically meaningful operators.

\color{black}
\begin{remark}
The existence or absence of obstructions is checked for each eigenvalue $h^{(j)}$, $j\in J$,  independently.  Suppose that the eigenvalues $h^{(l)}$ and $h^{(k)}$, $l\ne k$, are unobstructed. Then the choice of a smooth global eigenvector $v^{(l)}$ is not affected by and does not affect the choice of a smooth global eigenvector $v^{(k)}$.  As soon as one can choose smooth global eigenvectors $v^{(j)}$ for all $j\in J$,  the results from \cite{diagonalization} provide an explicit algorithm for the construction of the full symbol of a pseudodifferential operator $U$ such that $U^*AU$ is diagonal, i.e., the direct sum of $m$ scalar operators acting in $L^2(M)$, modulo an integral operator with infinitely smooth kernel.
\end{remark}

\begin{remark}
Some further comments are in order on the importance of achieving a \emph{global} diagonalisation.  Indeed,  in cases where Question~\ref{Q2} has a negative answer,  one may still pursue a local (or even microlocal) diagonalisation of the operator $A$. 
Unfortunately,  most of the time the latter is of little or no use in applications to, for example, spectral theory.  It was shown in \cite{diagonalization} that, in the absence of topological obstructions,  the spectrum of an operator $A$ of positive order is asymptotically well approximated by the union of the spectra of the scalar elliptic operators appearing on the `diagonal' of $U^*AU$, up to a superpolynomial error.
No such results can be established by means of mere local diagonalisation. There are,  however,  limited instances where one only needs a local diagonalisation for the spectral analysis, see for example \cite{grigori1}; this justifies addressing Questions~\ref{Q1} and~\ref{Q2} separately.

When available, diagonalisability of a system substantially simplifies the spectral analysis \cite{NP} and the construction of evolution operators \cite{dirac, lordir} by reducing the system to scalar operators. However, diagonalising a system may not be possible due to topological obstructions, as is the case for some important physically meaningful operators; see Sections~\ref{Local obstructions}--\ref{Neither local nor global obstructions}. For this reason, other approaches to \textcolor{black}{the study of} the spectrum of systems, such as the use of pseudodifferential projections \cite{part1, part2}, are perhaps more natural, in that they always work and circumvent topological obstructions altogether.
 \end{remark}

\color{black}



\section{\color{black}Examples: local obstructions}
\label{Local obstructions}

\subsection{Massless Dirac operator in 3D}
\label{Massless Dirac operator in 3D}

Let $(M,g)$ be a closed oriented Riemannian 3-manifold. We denote by $\nabla$ the Levi--Civita connection, by $\Gamma^\alpha{}_{\beta\gamma}$ the Christoffel symbols, and by $\rho(x):=\sqrt{\det g_{\alpha\beta}}$ the Riemannian density.

Let $\{e_j\}_{j=1}^3$ be a positively oriented global framing of $M$, namely,  a set of three orthonormal smooth vector fields on $M$, whose orientation agrees with that of $M$. Recall that such a global framing exists because all orientable 3-manifolds are parallelizable \cite{Stiefel, Kirby}. In chosen local coordinates $x^\alpha$, $\alpha=1,2,3$,  we will denote by $e_j{}^\alpha$ the $\alpha$-th component of the $j$-th vector field.  Let
\begin{equation*}
\sigma^\alpha(x):= \sum_{j=1}^3 s^j \,e_j{}^\alpha(x)
\end{equation*}
be the projection of the standard Pauli matrices
\smallskip\noindent
\begin{equation}\label{Pauli matrices basic}
s^1:=
\begin{pmatrix}
0&1\\
1&0
\end{pmatrix}
=s_1
\,,
\quad
s^2:=
\begin{pmatrix}
0&-i\\
i&0
\end{pmatrix}
=s_2
\,,
\quad
s^3:=
\begin{pmatrix}
1&0\\
0&-1
\end{pmatrix}
=s_3
\end{equation}
along our framing.

The massless Dirac operator acting on the sections of a trivial $\mathbb{C}^2$--bundle over $M$ is the $2\times 2$ differential operator defined by
\begin{equation}
\label{massless dirac definition equation}
W:=-i\sigma^\alpha
\left(
\frac\partial{\partial x^\alpha}
+\frac14\sigma_\beta
\left(
\frac{\partial\sigma^\beta}{\partial x^\alpha}
+\Gamma^\beta{}_{\alpha\gamma}\,\sigma^\gamma
\right)
\right)
:H^1(M;\mathbb{C}^2)\to L^2(M;\mathbb{C}^2).
\end{equation}
The operator \eqref{massless dirac definition equation} is an elliptic self-adjoint differential operator of order $1$. Its principal symbol reads
\begin{equation}
\label{principal symbol Dirac}
W_\mathrm{prin}(x,\xi)=\sigma^\alpha(x)\,\xi_\alpha.
\end{equation}
Furthermore,  a straightforward calculation involving elementary properties of Pauli matrices gives us
the eigenvalues 
\begin{equation*}
h^{(\pm)}(x,\xi)=\pm \sqrt{g^{\alpha\beta}(x)\xi_\alpha\xi_\beta}\,.
\end{equation*}
Hence,  the operator $W$ satisfies Assumption~\ref{assumption simple eigenvalues}.

\begin{proposition}
\label{proposition Dirac 3D}
Fix a point $x\in M$. It is impossible to choose eigenvectors $v^{(\pm)}(x,\xi)$ of \eqref{principal symbol Dirac} smoothly for all $\xi\in T^*_x M\setminus\{0\}$.
\end{proposition}

\begin{proof}
Let us choose geodesic normal coordinates centred at $x \in M$ in such a way that $e_j{}^\alpha(0) = \delta_j{}^\alpha$.  The latter can always be achieved by a rigid rotation of the coordinate system.  

Arguing by contradiction,  suppose $\widetilde v^{(+)}(\xi):=v^{(+)}(0,\xi)$ is defined for all $\xi$ in a smooth manner.  The eigenvectors $\widetilde v^{(+)}(\xi)$ are normalised, $[\widetilde v^{(+)}(\xi)]^*\widetilde v^{(+)}(\xi)=1$, and satisfy
\begin{equation}
\label{example 3D Dirac equation 1}
s^\alpha\xi_\alpha \, \widetilde v^{(+)}=|\xi| \,\widetilde v^{(+)},
\end{equation}
where $|\xi|$ is the Euclidean norm.  Multiplying \eqref{example 3D Dirac equation 1} by $[\widetilde v^{(+)}]^{\top} \epsilon$ \textcolor{black}{from the left}, where $\epsilon$ \textcolor{black}{is} the \textcolor{black}{`}metric\textcolor{black}{'} spinor
\begin{equation}
\label{example 3D Dirac equation 2}
\epsilon:=\begin{pmatrix}
0 & 1\\
-1 & 0
\end{pmatrix},
\end{equation}
cf.~\cite[Appendix~A.2]{diffeo},
we obtain
\begin{equation}
\label{example 3D Dirac equation 3}
u^\alpha \xi_\alpha=0,
\end{equation}
with $u^\alpha:=[\widetilde v^{(+)}]^{\top} \,\epsilon s^\alpha\, \widetilde v^{(+)}$.
The complex 3-vector $u$ is isotropic,
\begin{equation*}
\label{example 3D Dirac equation 4}
u^{\top} u=0,
\end{equation*}
see~\cite[Chapter III, Section I]{cartan},  has norm
\begin{equation}
\label{example 3D Dirac equation 5}
\|u\|=\sqrt{2}
\end{equation}
and is invariant under rigid rotations of the normal coordinate system.
Put $w:=\operatorname{Re}u$. Then formulae \eqref{example 3D Dirac equation 3}--\eqref{example 3D Dirac equation 5} imply
\begin{equation*}
\label{example 3D Dirac equation 6}
w^\alpha\xi_\alpha=0,  \qquad \|w\|=1.
\end{equation*}
This formula provides a nowhere zero tangent vector field $w$ on the 2-sphere $|\xi|=1$, which contradicts the hairy ball theorem; see for instance \cite[Theorem 2.28]{hatcher}.
\end{proof}

\color{black}
\subsection{A topological proof}\label{Top proof}
The contradiction argument above provides an analytic proof for the failure of the `topological' condition (2) from Theorem~\ref{main theorem 2} for the Dirac operator in 3D.  In this subsection we give an alternative proof of Proposition~\ref{proposition Dirac 3D}
\color{black}
for the special the case $M=\mathbb{S}^3$,  \textcolor{black}{one} relying directly on Theorem~\ref{main theorem 2}. 

View the round $\mathbb{S}^3$ as the Lie group $SU(2)$ with the bi-invariant metric and the unique spin structure. Identify $T_{\textcolor{black}{E}} (SU(2))$ with the Lie algebra $\mathfrak {su}(2)$ and the unit sphere \textcolor{black}{${\mathbb S}_E (SU(2)) \subset T_E (SU(2))\setminus\{0\}$} with the conjugacy class of zero-trace matrices in $SU(2)$. \textcolor{black}{Here $E \in SU(2)$ is the identity matrix in the Lie group $SU(2)$.} Every matrix in ${\mathbb S}_E(SU(2))$ is of the form 
\begin{equation}\label{E:conj}
\textcolor{black}{B}\, 
\begin{pmatrix}
-i & 0 \\
\;\; 0 & i 
\end{pmatrix}
\textcolor{black}{B}^{-1}
\end{equation}
for some $\textcolor{black}{B} \in SU(2)$ defined uniquely up to the right multiplication by a matrix in $U(1) \subset SU(2)$. This provides for the identification $\textcolor{black}{{\mathbb S}_E}(SU(2)) = SU(2)/U(1)$. Now, the \textcolor{black}{principal} symbol of the Dirac operator at $(\textcolor{black}{E},\zeta) \in \textcolor{black}{{\mathbb S}_E}(SU(2))$ is given by the Clifford multiplication $i\zeta: \mathbb C^2 \to \mathbb C^2$. It is clear from the above description of $\textcolor{black}{{\mathbb S}_E} (SU(2))$ that the eigenvalues of $i\zeta$ are $\lambda = \pm 1$. Let $\lambda = 1$; \textcolor{black}{the case of $\lambda = -1$ is similar.} Then the map \eqref{E:cp} sends the matrix \eqref{E:conj} to the equivalence class of the vector 
\[
\textcolor{black}{B}
\begin{pmatrix}
1 \\
0
\end{pmatrix}
\]
in $\CP^1$. With the identification $\textcolor{black}{{\mathbb S}_E} (SU(2)) = SU(2)/U(1)$, one can easily check that this map is  obtained from the Hopf map $SU(2) \to \CP^1$ by factoring out $U(1) \subset SU(2)$. \textcolor{black}{Therefore, the restriction of \eqref{E:cp} to ${\mathbb S}_E(SU(2))$ is a homeomorphism and, in particular, \eqref{E:cp} is not homotopic to a constant map.}

\subsection{The operator curl}
\label{The operator curl}

Let $(M,g)$ be a closed oriented Riemannian 3-manifold.  Let $\Omega^k(M)$ be the Hilbert space of real-valued $k$-forms over $M$, $k=1,2$.  We define the operator curl as
\begin{equation*}
\operatorname{curl}:= \ast d : \Omega^1(M) \to \Omega^1(M),
\end{equation*}
where $\ast$ is the Hodge dual and $d$ denotes the exterior derivative.  Note that $T^*M$ is trivial, see second paragraph in subsection~\ref{Massless Dirac operator in 3D}.

The operator curl is formally self-adjoint with respect to the natural inner product on $\Omega^1(M)$,  but not elliptic.  Indeed,  its principal symbol reads
\begin{equation}
\label{principal symbol curl}
[(\operatorname{curl})_\mathrm{prin}]_\alpha{}^\beta(x,\xi)= -i\,E_\alpha{}^{\beta\gamma}(x)\,\xi_\gamma\,,
\end{equation}
where the tensor $E$ is defined in accordance with
\begin{equation*}
E_{\alpha\beta\gamma}(x):=\rho(x)\,\varepsilon_{\alpha\beta\gamma},
\end{equation*}
$\rho$ being the Riemannian density and $\varepsilon$ the total antisymmetric symbol, $\varepsilon_{123}=+1$.
An elementary calculation tells us that $(\operatorname{curl})_\mathrm{prin}$ has the simple eigenvalues
\begin{equation*}
h^{(0)}(x,\xi)=0, \qquad h^{(\pm)}(x,\xi)=\pm\sqrt{g^{\mu\nu}(x)\,\xi_\mu\xi_\nu}=:\pm\|\xi\|_{g(x)}\,.
\end{equation*}
It follows that
\begin{equation*}
\det(\operatorname{curl})_\mathrm{prin}=0.
\end{equation*}

\begin{proposition}
\label{proposition local obstructions curl}
Fix a point $x\in M$. It is impossible to choose eigenvectors $v^{(\pm)}(x,\xi)$ of \eqref{principal symbol curl} smoothly for all $\xi\in T^*_x M\setminus\{0\}$.
\end{proposition}
\begin{proof}
Arguing by contradiction, suppose one can choose normalised eigenvectors $v^{(\pm)}(x,\xi)$ of \eqref{principal symbol curl} smoothly for all $\xi\in T^*_x M\setminus\{0\}$,
\begin{equation}
\label{proof local obstructions curl eqution 1}
g^{\alpha\beta}(x)\,\overline{[v^{(\pm)}(x,\xi)]_\alpha}\, [v^{(\pm)}(x,\xi)]_\beta=1.
\end{equation}
Here the overline denotes complex conjugation.

Let us begin by observing that $(\operatorname{curl})_\mathrm{prin}$ is antisymmetric, that is,
\begin{equation*}
\label{proof local obstructions curl eqution 2}
g_{\gamma\beta}(x)\,[(\operatorname{curl})_\mathrm{prin}]_\alpha{}^\gamma(x,\xi)=-g_{\gamma\alpha}(x)\,[(\operatorname{curl})_\mathrm{prin}]_\beta{}^\gamma(x,\xi).
\end{equation*}
It ensues that $v^{(\pm)}(x,\xi)$ are complex isotropic 3-vectors:
\begin{equation}
\label{proof local obstructions curl eqution 3}
g^{\alpha\beta}(x)\,[v^{(\pm)}(x,\xi)]_\alpha\, [v^{(\pm)}(x,\xi)]_\beta =0,
\end{equation}
compare with \eqref{proof local obstructions curl eqution 1}. Put
\begin{equation*}
\label{proof local obstructions curl eqution 4}
w^{(\pm)}(\xi):=\operatorname{Re} v^{(\pm)}(x,\xi).
\end{equation*}
 Formulae \eqref{proof local obstructions curl eqution 1} and \eqref{proof local obstructions curl eqution 3} imply
\begin{equation}
\label{proof local obstructions curl eqution 5}
\overline{[w^{(\pm)}(\xi)]_\alpha} \,[w^{(\pm)}(\xi)]^\alpha=1/2.
\end{equation}
Multiplying the eigenvalue equation
\begin{equation*}
\label{proof local obstructions curl eqution 6}
-i\,E_\alpha{}^{\beta\gamma}(x)\,\xi_\gamma [v^{(\pm)}(x,\xi)]_\beta=\pm \|\xi\|_{g(x)}\,[v^{(\pm)}(x,\xi)]_\alpha
\end{equation*}
for $(\operatorname{curl})_\mathrm{prin}$ (recall~\eqref{principal symbol curl}) by 
$
E^{\alpha\mu\nu}\xi_\mu(x) [v^{(\pm)}(x,\xi)]_\nu
$
one obtains
\begin{equation*}
\label{proof local obstructions curl eqution 7}
\xi^\beta \, [v^{(\pm)}(x,\xi)]_\beta=0,
\end{equation*}
which,  in turn, yields
\begin{equation}
\label{proof local obstructions curl eqution 8}
\xi^\beta \,[w^{(\pm)}(\xi)]_\beta=0.
\end{equation}

Formulae \eqref{proof local obstructions curl eqution 5} and \eqref{proof local obstructions curl eqution 8} imply that  $w^{(\pm)}$ are, modulo scaling,  nowhere vanishing real vector fields tangent to the 2-sphere.  This contradicts the hairy ball theorem.
\end{proof}

\section{\color{black}Examples: global obstructions but no local obstructions}
\label{Global obstructions but no local obstructions}

\subsection{Restriction of the massless Dirac operator to the 2-sphere}
\label{Restriction of the massless Dirac operator to the 2-sphere}

Let $\mathbf{x}^\alpha$,  $\alpha=1,2,3$,  be the Euclidean coordinates in $\mathbb{R}^3$ and consider the massless Dirac operator $\mathbf{W}$ on $\mathbb{R}^3$ associated with the framing $\mathbf{e}_j{}^\alpha(\mathbf{x})=\mathbf{\delta}_j{}^\alpha$, $j,\alpha=1,2,3$,
\begin{equation}
\label{example 2D Dirac equation 1}
\mathbf{W}=-i s^\alpha
\frac\partial{\partial \mathbf{x}^\alpha}.
\end{equation}
Here the $s^\alpha$ are the standard Pauli matrices \eqref{Pauli matrices basic}.

Let $W$ be the restriction of \eqref{example 2D Dirac equation 1} to 
\begin{equation*}
\label{example 2D Dirac equation 2}
M=\mathbb{S}^2:=\{\mathbf{x}\in \mathbb{R}^3 \ |\ (\textbf{x}^1)^2+(\textbf{x}^2)^2+(\textbf{x}^3)^2=1 \}
\end{equation*}
equipped with the standard round metric. Throughout this subsection, we \textcolor{black}{use bold font to denote} quantities living in $\mathbb{R}^3$, to distinguish them from quantities living on $\mathbb{S}^2$.

The principal symbol of $W$ can be written explicitly in terms of 3-dimensional quantities as
\begin{equation}
\label{example 2D Dirac equation 3}
W_\mathrm{prin}(\mathbf{x}, \boldsymbol{\xi})=
\begin{pmatrix}
\boldsymbol{\xi}_3
&
\boldsymbol{\xi}_1-i\boldsymbol{\xi}_2
\\
\boldsymbol{\xi}_1+i\boldsymbol{\xi}_2
&
-\boldsymbol{\xi}_3
\end{pmatrix},
\end{equation}
\textcolor{black}{where $\mathbf{x}$ and $\boldsymbol{\xi}$ are subject to the conditions}
\begin{gather}
\textcolor{black}{(\textbf{x}^1)^2+(\textbf{x}^2)^2+(\textbf{x}^3)^2=1}, \label{example 2D Dirac equation 4} \\
\textcolor{black}{\textbf{x}^1\boldsymbol{\xi}_1\,+\,\textbf{x}^2\boldsymbol{\xi}_2\,+\,\textbf{x}^3\boldsymbol{\xi}_3 = 0.} \label{example 2D Dirac equation 5}
\end{gather}

\begin{proposition}\label{proposition 2D dirac}
It is impossible to choose eigenvectors $v^{(\pm)}(\mathbf{x},\boldsymbol{\xi})$ of \eqref{example 2D Dirac equation 3} smooth\-ly for all $(\mathbf{x},\boldsymbol{\xi}) \in T^*\mathbb{S}^2\setminus\{0\}$.
\end{proposition}

\begin{proof}
\textcolor{black}{Let $S^*(\mathbb{S}^2) \subset T^*\mathbb{S}^2\setminus\{0\}$ be the unit sphere bundle cut out by the equation}
\begin{equation} \label{example 2D Dirac equation 6}
\textcolor{black}{(\boldsymbol{\xi}_1)^2+(\boldsymbol{\xi}_2)^2+(\boldsymbol{\xi}_3)^2\, =1.}
\end{equation}
One can easily see that the eigenvalues of \eqref{example 2D Dirac equation 3} \textcolor{black}{on $S^*(\mathbb{S}^2)$} are
$h^{(\pm)}(\mathbf{x},\boldsymbol{\xi})=\pm 1$. Let $P^{(\pm)}(\mathbf{x},\boldsymbol{\xi})$ be the eigenprojections of $W_\mathrm{prin}(\mathbf{x}, \boldsymbol{\xi})$ corresponding to the eigenvalues $\pm 1$. Then a straightforward calculation shows that the maps \eqref{E:cp} are given by
\begin{equation}\label{E:two maps}
T^*\mathbb{S}^2\setminus\{0\} \longrightarrow \CP^1 = \mathbb {S}^2, \qquad (\mathbf{x}, \boldsymbol{\xi})\mapsto \operatorname{tr}(s^\alpha P^{(\pm)}(\mathbf{x},\boldsymbol{\xi}))=\pm\boldsymbol{\xi}^\alpha,
\end{equation}
supplemented by conditions \eqref{example 2D Dirac equation 4}--\eqref{example 2D Dirac equation 6}. Because of the symmetry between $\mathbf{x}$ and $\boldsymbol{\xi}$, \textcolor{black}{the maps \eqref{E:two maps} can be viewed as the bundle projection $f: S^* (\mathbb{S}^2) \longrightarrow \mathbb{S}^2$. By identifying $S^*(\mathbb{S}^2)$ with the real projective space $\RP^3$, we conclude that $H^2(S^* (\mathbb{S}^2)) = H^2 (\RP^3) = \mathbb Z/2$ and the induced map $f^*: H^2 (\mathbb S^2) \longrightarrow H^2 (S^*(\mathbb{S}^2))$ in cohomology is the mod 2 homomorphism $f^*: \mathbb Z \to \mathbb Z/2$. The latter can be seen from the Gysin exact sequence of the circle bundle $f: \RP^3 \longrightarrow \mathbb{S}^2$,}
\[
\begin{CD}
0 @>>> H^0 (\mathbb{S}^2) @>>> H^2 (\mathbb{S}^2) @> p^* >> H^2 (\RP^3) @>>> 0.
\end{CD}
\]

\smallskip\noindent
\textcolor{black}{Since $S^*(\mathbb{S}^2)$ is a deformation retract of $T^*\mathbb{S}^2\setminus\{0\}$ , it follows from Theorem \ref{main theorem 1} that we have a non-trivial obstruction to the existence of a smooth eigenvector $v^{(j)}(\mathbf{x}, \boldsymbol{\xi})$ for all $(\mathbf{x}, \boldsymbol{\xi}) \in T^*\mathbb{S}^2\setminus\{0\}$. Note that $d = 2$ in this case, hence the existence of $v^{(j)}(\mathbf{x}, \boldsymbol{\xi})$ for any fixed $\mathbf{x} \in \mathbb{S}^2$ is unobstructed by Theorem \ref{main theorem 2}.}
\end{proof}

%

\subsection{An artificial example}
\label{An artificial example}

In the same setting and with the same notation as in the previous subsection,  let us define
\begin{equation*}
\label{artificial example equation 1}
P_+(\mathbf{x},\boldsymbol{\xi})
:=
\frac{1}{2}
(
s_\alpha\mathbf{x}^\alpha
+
\textcolor{black}{E}
),
\end{equation*}
\begin{equation*}
\label{artificial example equation 2}
P_-(\mathbf{x},\boldsymbol{\xi})
:=
-
\frac{1}{2}
(
s_\alpha\mathbf{x}^\alpha
-
\textcolor{black}{E}
),
\end{equation*}
where $\textcolor{black}{E}$ is the $2\times 2$ identity matrix.  Using elementary properties of Pauli matrices it is easy to see that
\begin{equation*}
\label{artificial example equation 2bis}
P_+(\mathbf{x},\boldsymbol{\xi})P_-(\mathbf{x},\boldsymbol{\xi})=0 \quad\text{\textcolor{black}{and}}\quad [P_\pm(\mathbf{x},\boldsymbol{\xi})]^2=P_\pm(\mathbf{x},\boldsymbol{\xi}).
\end{equation*}

\noindent
Let us consider an elliptic pseudodifferential operator $A$ of order $s$ on $\mathbb{S}^2$ with principal symbol
\begin{equation}
\label{artificial example equation 3}
A_\mathrm{prin}(\mathbf{x},\boldsymbol{\xi})
:=
\|\boldsymbol{\xi}\|^s
[
c_+
P_+(\mathbf{x},\boldsymbol{\xi})
+
c_-
P_-(\mathbf{x},\boldsymbol{\xi})
],
\qquad c_+\ne c_-\,,
\qquad c_\pm\ne0.
\end{equation}

\begin{proposition}
It is impossible to choose eigenvectors $v^{(j)}(\mathbf{x},\boldsymbol{\xi})$, $j\in J$,  of \eqref{artificial example equation 3} smoothly for all $(\mathbf{x},\boldsymbol{\xi})$ satisfying  \eqref{example 2D Dirac equation 4}--\eqref{example 2D Dirac equation 6}.
\end{proposition}
\begin{proof}
Arguing as in the proof of Proposition~\ref{proposition 2D dirac}, we obtain that the maps \eqref{E:cp} are given by
\begin{equation*}
T^*\mathbb{S}^2\setminus\{0\} \longrightarrow \CP^1 = \mathbb{S}^2, \qquad (\mathbf{x}, \boldsymbol{\xi})\mapsto \operatorname{tr}(s^\alpha P_\pm(\mathbf{x},\boldsymbol{\xi}))=\pm\mathbf{x}^\alpha,
\end{equation*}
supplemented by conditions \eqref{example 2D Dirac equation 4}--\eqref{example 2D Dirac equation 6}. \textcolor{black}{These maps, once again, arise from} the cotangent bundle projection $T^* \mathbb{S}^2 \setminus \{0\} \longrightarrow \mathbb{S}^2$, only we now project to `position', as opposed to `momentum'. The same proof applies.
\end{proof}

\section{\color{black}Examples: neither local nor global obstructions}
\label{Neither local nor global obstructions}

\subsection{Linear elasticity in 2D}
\label{Linear elasticity}

Let $M$ be the 2-torus $\mathbb{T}^2$ endowed with a Riemannian metric $g$.  The operator of linear elasticity $L$ on vector fields
is defined in accordance with
\begin{equation}
\label{L}
(Lv)^\alpha=-\mu(\nabla_\beta \nabla^\beta v^\alpha+\operatorname{Ric}^\alpha{}_\beta \,v^\beta)-(\lambda+\mu)\nabla^\alpha\nabla_\beta v^\beta,
\end{equation}
where $\nabla$ is the Levi-Civita connection,  $\operatorname{Ric}$ is the Ricci tensor, and the \textcolor{black}{real} scalars $\lambda$ and $\mu$ are the \emph{Lam\'e parameters}.  The Lam\'e parameters are assumed to satisfy the conditions
\begin{equation}
\label{strong convexity}
\mu>0, \qquad \lambda+\mu>0,
\end{equation}
which guarantee strong convexity;  see for instance \cite{miyanishi}.  Formula \eqref{L} is obtained by performing an integration by parts in the identity
\begin{equation*}
\label{definition of L variational form}
\frac12 \int_M g_{\alpha\beta}\,v^\alpha (Lv)^\beta \rho\, dx=E(v),
\end{equation*}
$E(v)$ being the potential energy of elastic deformation
\begin{equation*}
\label{potential energy of elastic deformation}
E(v):=\frac{1}{2}\int_M \left(\lambda(\nabla_\alpha v^\alpha)^2+\mu(\nabla_\alpha v_\beta+\nabla_\beta v_\alpha)\nabla^\alpha v^\beta \right) \rho\,dx
\end{equation*}
and $\rho(x):=\sqrt{\operatorname{det}g_{\alpha\beta}(x)}$ being the Riemannian density.  In the presence of a boundary, the latter supplies appropriate boundary conditions. A detailed derivation can be found, for example, in \cite{MR0106584,diffeo}.

The operator $L$, which acts on 2-vectors, can be turned into an operator acting on $2$-columns of scalar functions as follows.  

\textcolor{black}{Recall that the torus $\mathbb{T}^2$ is parallelizable. Choose a global orthonormal framing $e_j$, $j=1,2$, on $\mathbb{T}^2$ and put}
\begin{equation*}
e^j{}_\alpha:=\delta^{jk}g_{\alpha\beta}\,e_k{}^\beta.
\end{equation*}
Define the operator $\textcolor{black}{\mathbf{S}}$,
\begin{equation*}
(\textcolor{black}{\mathbf{S}}v)^j:=e^j{}_\alpha v^\alpha\,,
\end{equation*}
which maps 2-vectors to 2-columns of scalar functions.
The operator of linear elasticity acting on 2-columns of scalar functions is defined as
\begin{equation*}
\label{L on half densities}
L_{\mathrm{scal}}:=\textcolor{black}{\mathbf{S}}L\textcolor{black}{\mathbf{S}}^{-1}.
\end{equation*}

A straightforward calculation gives
\begin{equation}
\label{L principal}
(L_{\mathrm{scal}})_\mathrm{prin}=\mu h^2\,I+(\lambda+\mu)h^2\, qq^{\top},
\end{equation}
where
\begin{equation}
\label{h and p}
h(x,\xi):=\sqrt{g^{\alpha\beta}(x)\xi_\alpha\xi_\beta}\,,
\qquad
q(x,\xi):=\frac{1}{h(x,\xi)}\begin{pmatrix}
e_1{}^\alpha(x)\, \xi_\alpha\\
e_2{}^\alpha(x)\, \xi_\alpha
\end{pmatrix}.
\end{equation}
Analysing \eqref{L principal} we conclude that the eigenvalues of $(L_{\mathrm{scal}})_\mathrm{prin}$ are
\begin{equation}
\label{eigenvalue L prin}
h^{(1)}=\mu h^2, \qquad h^{(2)}=(\lambda+2\mu)h^2
\end{equation}
and the corresponding orthonormalised eigenvectors are
\begin{equation}
\label{eigenvectors L prin}
v^{(1)}=\epsilon \,q\,, \qquad v^{(2)}=q.
\end{equation}
Recall that $\epsilon$ is defined in accordance with \eqref{example 3D Dirac equation 2}. Note that the eigenvalues \eqref{eigenvalue L prin} are simple; indeed,  conditions \eqref{strong convexity} imply $h^{(2)}/h^{(1)}>1$.

It ensues that $L_{\mathrm{scal}}$ satisfies \textcolor{black}{Assumptions \ref{assumption principal symbol} and \ref{assumption simple eigenvalues}} from Section~\ref{Statement of the problem}. Formulae \eqref{h and p} and \eqref{eigenvectors L prin} imply that $v^{(1)}(x,\xi)$ and $v^{(2)}(x,\xi)$ are smoothly defined for all $(x,\xi)\in T^*(\mathbb{T}^2)\setminus \{0\}$.

\begin{remark}
Let us point out that this is not a trivial example: there exist systems of two equations on $\mathbb{T}^2$ topologically obstructed as per Theorem~\ref{main theorem 1}. Indeed,  there exist maps $S^* (\mathbb T^2) = \mathbb T^3 \to \CP^1$ that induce non-zero homomorphisms $H^2 (\CP^1) \to H^2 (\mathbb T^3)$. To obtain an example, simply compose the projection map $\mathbb T^3 = \mathbb T^2 \times \mathbb S^1 \longrightarrow \mathbb T^2$ with any map $\mathbb T^2 \longrightarrow \textcolor{black}{\CP^1}$ of degree one. 
\end{remark}

%


\subsection{The Neumann--Poincar\'e operator for linear elasticity in 3D}
\label{NP example}

Let $\mathcal{D}$ be \textcolor{black}{a} bounded connected domain of $\mathbb{R}^3$ with smooth closed boundary $M$ and let $g$ be the Riemannian metric on $M$ induced by the standard Euclidean metric on $\mathbb{R}^3$.  We denote by $\textbf{x}=(\textbf{x}^1,\textbf{x}^2,\textbf{x}^3)$ the standard Euclidean coordinates in $\mathbb{R}^3$.

The operator of linear elasticity $L$ acting on vector fields in $\mathcal{D}$  is defined in accordance with
\begin{equation*}
\label{L NP}
(Lv)^\alpha:=-\mu\,\partial_\beta \partial^\beta v^\alpha-(\lambda+\mu)\partial^\alpha\partial_\beta v^\beta,
\end{equation*}
where the scalars $\lambda$ and $\mu$, assumed to satisfy the conditions
\begin{equation*}
\label{strong convexity NP}
\mu>0, \qquad \lambda+\frac23\mu>0,
\end{equation*}
are the Lam\'e parameters. Compare with \eqref{L} and \eqref{strong convexity}.

The \emph{Kelvin matrix}
\begin{equation*}
\label{Kelvin matrix}
[\mathcal{K}(\textbf{x},\textbf{y})]^\alpha{}_\beta:=\frac{\lambda+3\mu}{4\pi\, \mu(\lambda+2\mu)} \frac{\delta^\alpha{}_\beta}{|\textbf{x}-\textbf{y}|}+\frac{\lambda+\mu}{4\pi\,\mu(\lambda+2\mu)} \frac{(\textbf{x}-\textbf{y})^\alpha(\textbf{x}-\textbf{y})_\beta}{|\textbf{x}-\textbf{y}|^3}
\end{equation*}
is related \textcolor{black}{to} the fundamental solution $\mathcal{E}$ of $L$ as
\begin{equation*}
\label{fundamental solution L}
\mathcal{E}(\textbf{x},\textbf{y})=\frac12\, \mathcal{K}(\textbf{x},\textbf{y}),
\end{equation*}
see~\cite[Eqns.~(1.23) and~(1.28)]{agranovich} (note that the operator of linear elasticity in \cite{agranovich} is defined to be $-L$).

We define the \emph{Neumann--Poincar\'e operator} to be the zeroth order pseudodifferential operator acting on vector fields on $M$ by the formula
\begin{equation}
\label{NP definition}
[\mathcal{B}v(\textbf{x})]^\alpha:=\int_M  \delta^{\alpha\mu}\,\delta_{\beta\nu}\,[T(\textbf{y}, \partial_{\textbf{y}})\mathcal{E}(\textbf{x},\textbf{y})]^\nu{}_\mu \,v^\beta(\textbf{y})\, dS, \qquad \textbf{x}\in M,
\end{equation}
with
\begin{equation*}
\label{traction}
T: u^\alpha \mapsto [T(\textbf{x}, \partial_{\textbf{x}})u(\textbf{x})]^\alpha:=\lambda n^\alpha(\textbf{x}) \partial_\beta u^\beta(\textbf{x})  +\mu\left(n^\beta(\textbf{x})  \partial_\beta u^\alpha(\textbf{x})  + n_\beta(\textbf{x})  \partial^\alpha u^\beta(\textbf{x}) \right)
\end{equation*}
known as the traction. Here $n$ denotes the outer unit normal vector field on $M$. The operator $\mathcal{B}$ is a singular integral operator, and \textcolor{black}{the integral in} formula \eqref{NP definition} is to be understood in the sense of Cauchy principal value. Note that $\mathcal{B}$ is neither elliptic nor self-adjoint in $L^2(M)$ \cite{miyanishi,NP}.

\begin{remark}
Let us point out that,  in this example, the tangent bundle $TM$ is not necessarily trivial. Hence,  this doesn't fully align with the framework set out in the beginning of the paper.  Nevertheless,  we analyse the issue of obstructions for the Neumann--Poincar\'e operator in this slightly more general setting because of its importance in applications.
\end{remark}

Let $x=(x^1, x^2)$ be an arbitrary \textcolor{black}{local} coordinate system on $M$.  Given a point
$\textbf{x}\textcolor{black}{\in\mathcal{D}}$ in a neighbourhood of $M$, we define $x^3(\textbf{x}):=\operatorname{dist}(\textbf{x}, M)$ to be its distance to $M$ and $x(\textbf{x})=\Pi_M(\textbf{x})$ to be its orthogonal projection onto $M$.  Then $(x=(x^1, x^2), x^3)$ defines a coordinate system in a neighbourhood of $M$. In this coordinate system, the principal symbol of $\mathcal{B}$ reads \cite[Eqn.~(1.89)]{agranovich}\footnote{Note that formula (1.89) in \cite{agranovich} has the opposite sign, because the authors there started from the operator $-L$ as opposed to $L$.}
\begin{equation}
\label{NP principal symbol}
\mathcal{B}_\mathrm{prin}(x,\xi)=-\frac{i\mu}{2(\lambda+2\mu)}\frac{1}{\sqrt{g^{\mu\nu}(x)\,\xi_\mu\xi_\nu}}
\begin{pmatrix}
0&-g^{\alpha\gamma}(x)\,\xi_\gamma\\
\xi_\beta&0
\end{pmatrix}, \qquad (x,\xi)\in T^*M\setminus\{0\}.
\end{equation}
The zero in the upper-left corner \textcolor{black}{of the matrix in} \eqref{NP principal symbol} is a $2\times 2$ block of zeros.
The principal symbol \eqref{NP principal symbol} acts on quantities of the form
\begin{equation*}
\begin{pmatrix}
w\\
f
\end{pmatrix},
\end{equation*}
where $w$ is a vector field on $M$ and $f$ is a scalar field on $M$.
A straightforward calculation shows that the eigenvalues of \eqref{NP principal symbol} are
\begin{equation}
\label{NP eigenvalues principal symbol}
h^{(0)}(x,\xi)=0, \qquad h^{(\pm)}(x,\xi)=\pm \frac{\mu}{2(\lambda+2\mu)}\,.
\end{equation}

\begin{theorem}
One can choose linearly independent orthonormal eigenvectors $v^{(0)}(x,\xi)$ and $v^{(\pm)}(x,\xi)$ of \eqref{NP principal symbol} corresponding to the eigenvalues \eqref{NP eigenvalues principal symbol} smoothly for all $(x,\xi)\in T^*M\setminus\{0\}$.
\end{theorem}

\begin{proof}
Since $\dim M\ne 3$, there are no local obstructions. \textcolor{black}{To sort out global obstructions, we will work with the unit sphere bundle $S^*M$, which is a deformation retract of $T^*M\setminus \{0\}$. By direct inspection, the eigenvectors $v^{(\pm)}$ corresponding to the nonzero eigenvalues give rise to the maps}
\begin{equation}
\label{proof NP equation 1}
f^{(\pm)}:S^*M \longrightarrow \CP^2, \qquad (x,\xi) \mapsto n(x)\pm i\xi,
\end{equation}
Formula \eqref{proof NP equation 1}, in fact, implies that one has well defined smooth eigenvectors in $\mathbb{S}^5$ and not just in $\CP^2$, which yields global existence.

Let us now examine \textcolor{black}{global} obstructions for the eigenvector $v^{(0)}$ with zero eigenvalue. Observe that $v^{(0)}$ can be chosen to be real hence $v^{(0)}$ gives rise to the map 
\[
f^{(0)}: S^*M \longrightarrow \RP^2
\]
sending $(x,\xi)$ to the line $L \subset T_x M$ perpendicular to $\xi \in T_x M$. A choice of complex structure\footnote{Here we are using the fact that an oriented closed two-dimensional surface is a Riemann surface,  hence it admits a complex structure, see \cite[Subsection~2.1]{donaldson}.} on $M$ gives us a preferred direction of rotation in each tangent plane $T_x M$, and hence a \textcolor{black}{consistent} choice of a specific unit vector on the line $L$. 
\end{proof}

\section{Pseudodifferential operators with multiplicities}
\label{Pseudodifferential operators with multiplicities}
In conclusion, we wish to mention that there are many pseudodifferential operators whose principal symbols have multiple eigenvalues. The list of such operators includes the Neumann--Poincar\'e operator in higher dimensions, the operator of linear elasticity in dimensions three and higher,  the signature operator,  Dirac operators in higher dimensions etc.  It would be interesting to investigate the diagonalisation question for these operators; here is a quick outline.

An eigenvalue of multiplicity $k \ge 1$ leads as before to a well-defined map $T^*M \setminus\{0\} \to \Gr_k (\mathbb C^m)$ to the Grassmanian of $k$-dimensional complex planes in $\mathbb C^m$. This map needs to be lifted to the canonical bundle $\V_k (\mathbb C^m) \to \Gr_k (\mathbb C^m)$, where $\V_k (\mathbb C^m)$ stands for the Stiefel manifold of $k$-frames in $\mathbb C^m$. The case $k = 1$ corresponds to the map \eqref{E:cp} and the canonical bundle \eqref{E:bundle} because  $\V_1 (\mathbb C^m) = \mathbb S^{2m-1}$ and $\Gr_1 (\mathbb C^m) = \CP^{m-1}$. 

The lifting problem at hand is obstructed by the higher Chern classes $c_1,\ldots, c_k$. This is consistent with the $k = 1$ case because the first Chern class of a complex line bundle coincides with the Euler class of the same bundle viewed as an oriented 2-plane bundle. Unlike in the $k = 1$ case, however, the Chern classes do not provide in general a full set of obstructions: there exist non-trivial complex bundles all of whose Chern classes vanish. A simple example of that is the $U(2)$ bundle over $\mathbb S^5$ with the clutching function $\mathbb S^4 \to U(2)$ representing the non-trivial element in $\pi_4 (U(2)) = \mathbb Z/2$.

The above discussion illustrates that the case of operators with multiplicities is quite different: one would not be able to obtain as sharp results in full generality (see also \cite{grove_pedersen,friedman_park} \textcolor{black}{and Remark~\ref{remark smoothness eigenvalues}}), and more of a \textcolor{black}{case-by-case} analysis would be required.  Hence,  we refrain from analysing operators with multiplicities in this paper.

\section*{Acknowledgements}
\addcontentsline{toc}{section}{Acknowledgements}

MC was partially supported by the Leverhulme Trust Research Project Grant RPG-2019-240 and by a Heilbronn Small Grant (via the UKRI/EPSRC Additional Funding Programme for Mathematical Sciences); both are gratefully acknowledged. GR was supported by a grant from Ministry of Science and Higher Education of RF, Agreement  075-15-2022-287. NS was partially supported by NSF Grant DMS-1952762.

\end{document}